\documentclass[12pt]{amsart}
\usepackage[colorlinks=true,pagebackref,hyperindex,citecolor=blue,linkcolor=red]{hyperref}
\usepackage{amsmath}
\usepackage{amsfonts}
\usepackage{setspace}
\usepackage{moreverb}
\usepackage[dvips]{graphicx}
\usepackage[latin1]{inputenc}
\usepackage[T1]{fontenc}
\usepackage{amssymb}
\usepackage{tikz}
\usepackage{color}
\usepackage[all]{xy}  
\usepackage{xfrac}
\usepackage[top=1in, bottom=1in, left=1in, right=1in]{geometry}
\usepackage{mathrsfs}

\theoremstyle{definition}
\newtheorem{theorem}{Theorem}[section]
\newtheorem*{maintheorem}{Theorem}

\newtheorem{corollary}[theorem]{Corollary}
\newtheorem{lemma}[theorem]{Lemma}
\newtheorem{proposition}[theorem]{Proposition}

\theoremstyle{definition}
\newtheorem{definition}[theorem]{Definition}

\newtheorem{example}[theorem]{Example}

\newtheorem{remark}[theorem]{Remark}
\numberwithin{equation}{subsection}

\newcommand{\NN}{\mathbb{N}}
\newcommand{\ZZ}{\mathbb{Z}}

\newcommand{\m}{\mathfrak{m}}

\newcommand{\cK}{\mathcal{K}}
\newcommand{\cF}{\mathcal{F}}

\newcommand{\cM}{\mathcal{M}}

\newcommand{\Spec}{\operatorname{Spec}}
\newcommand{\Tor}{\operatorname{Tor}}
\newcommand{\rk}{\operatorname{rank}}
\newcommand{\Hom}{\operatorname{Hom}}
\newcommand{\Ext}{\operatorname{Ext}}

\newcommand{\Ker}{\operatorname{Ker}}

\newcommand{\Length}{\operatorname{length}}

\newcommand{\Max}{\operatorname{max}}
\newcommand{\Ht}{\operatorname{ht}}

\newcommand{\Cech}{ \check{\rm{C}}}

\newcommand{\FDer}[1]{\stackrel{#1}{\to}}

\newcommand\wS{\widetilde{S}}
\newcommand\wg{\widetilde{g}}

\newcommand{\ba}{\alpha}
\newcommand{\bb}{\beta}

\newcommand{\be}{\mathbf{e}}
\newcommand{\gmod}{\operatorname{* mod}}
\newcommand{\Sq}{\operatorname{Sq}}

\newcommand\kk{\Bbbk}

\newcommand\pol{\operatorname{\mathsf{pol}_\ba}}
\newcommand\dpol{\operatorname{\mathsf{pol}^\ba}}
\newcommand\wR{\widetilde{R}}
\newcommand\wI{\widetilde{I}}
\newcommand{\fs}{\mathfrak{s}}
\newcommand{\fr}{\mathfrak{r}}
\newcommand{\zero}{\mathbf{0}}
\newcommand{\supp}{\mbox{\rm{supp}} }
\newcommand{\bD}{\mathbf{D}}

\newcommand{\bN}{{\mathbb{N}}}
\newcommand{\bZ}{{\mathbb{Z}}}

\newcommand{\gs}{\sigma}
\newcommand\rc{{\rm c}}
\newcommand{\fl}{\mathfrak{l}}
\newcommand\lk{{\mathrm{lk}}}
\newcommand{\fp}{\mathfrak{p}}
\newcommand{\fn}{\mathfrak{n}}
\newcommand{\fq}{\mathfrak{q}}

\begin{document}
\newcommand{\tens}{\otimes}
\newcommand{\hhtest}[1]{\tau ( #1 )}
\renewcommand{\hom}[3]{\operatorname{Hom}_{#1} ( #2, #3 )}

\title{Properties of Lyubeznik numbers under localization and polarization}
\author{Arindam Banerjee}
\author{Luis N\'u\~nez-Betancourt}
\author{Kohji Yanagawa}
\maketitle

\begin{abstract} 
We exhibit a global bound for the Lyubeznik numbers of a ring of prime characteristic. In addition, we show that for a monomial ideal, the Lyubeznik numbers of the quotient rings of its radical and its polarization are the same. Furthermore, we present examples that show striking behavior of the Lyubeznik numbers under localization. We also show related results for generalizations of the Lyubeznik numbers.
\end{abstract}

%%%%%%%%%%%%%%%%%%%%%%%%%%%%%%%%%%%%%%%%%%%%%%%%%%%%%%%%%%%
\section{Introduction}
%%%%%%%%%%%%%%%%%%%%%%%%%%%%%%%%%%%%%%%%%%%%%%%%%%%%%%%%%%%
In $1993$ Lyubeznik \cite{LyuDMod} introduced a family of invariants for a local ring containing a field, $R,$ today called Lyubeznik numbers and denoted by $\lambda_{i,j}(R)$ (see Section \ref{Sec LC-LyuNum}). These numbers
have been shown to have multiple connections; for instance, they relate to  singular and \`etale cohomology \cite{B-B,GarciaSabbah,LyuDMod}, to the Hochster-Huneke graph \cite{LyuInvariants,W}, and to projective varieties \cite{WZ-projective,Switala}.
These connections have motivated multiple generalizations; for instance, for mixed characteristic rings \cite{NuWiMixChar}, and rings of equal-characteristic from a differential perspective \cite{AM-Proc,NuWi1}.

In this article, we study the behavior of these invariants under
localization. 
The first result obtained in this context is that the Lyubeznik numbers are bounded globally over rings of positive characteristic, which resemble  behavior of Bass numbers under localization (see \ref{BoundFG}).  

\begin{maintheorem}[see Theorem \ref{Main Localization}]
Let $R$ be a ring which is a quotient of a regular Noetherian regular ring of finite dimension and positive characteristic $p>0.$
Then, there exists a positive integer $B$ such that
$$
\lambda_{i,j}(R_{\fp})\leq B
$$
for every $i,j\in\NN$ and ${\fp}\in\Spec(R).$
\end{maintheorem}

The previous claim was proven for algebras finitely generated over a field of characteristic zero or an algebraically closed field of prime characteristic  by Puthenpurakal \cite{TonyInjRes}.
We point out that his result deals only with localization at maximal ideals.

Unfortunately, the Lyubeznik numbers do not behave much better than it is stated in the previous theorem. We show an example of a Stanley-Reisner ring whose  highest Lyubeznik number could either decrease or increase under localization (see Example \ref{ExampleBadB}). In fact, we show a method to build a Stanley-Reisner ring with  all but one Lyubeznik number vanishing and with a localization that, surprisingly,  have many positive Lyubeznik numbers (see Remark \ref{Grow with Loc}). 

An operation related with localization of Stanley-Reisner rings is polarization (see Remark \ref{RelLocPol}). Given a monomial ideal $I$, not necessarily radical, in a polynomial ring $S$, 
we consider the polarization ideal $\widetilde{I}$ in
the polarization ring $\widetilde{S}.$
We compare the Lyubeznik number at the maximal homogeneous ideal of $S/\sqrt{I}$ 
and $\widetilde{S}/\widetilde{I}.$

\begin{maintheorem}[see Theorem \ref{Main Thm Pol LyuNum}]
Let $S=K[x_1,\ldots,x_n]$ be a polynomial ring, $\m=(x_1,\ldots,x_n)$ and $I\subset S$ be a monomial ring. Let $\widetilde{I}$ denote the polarization of $I,$ and  $\widetilde{S}=K[x_{r,s}]$ denote the polarization ring. Let $\fn=(x_{i,j}),$
and  $h=\dim(\widetilde{S}/\widetilde{I})-\dim(S/I)$.
Then,
$$
\lambda_{i-h,j-h}\left(S_\m/IS_\m\right)= \lambda_{i,j}\left(\widetilde{S}_\fn/\widetilde{I}\widetilde{S}_\fn\right).
$$
for every $i,j\in\NN.$
\end{maintheorem}

This result gives surprisingly different behavior of Lyubeznik numbers under localization and polarization. 
We also study the behavior of the
the generalized Lyubeznik numbers, $\lambda^0_i(R)$, for rings of equal characteristic under localization and polarization for Stanley-Reisner rings (see Propositions \ref{Prop Loc Gral LyuNum} and \ref{Prop Pol Gral LyuNum}). 
As a consequence of our methods, we observe similar behavior for  finer invariants given by the multiplicities of characteristics cycles.

%%%%%%%%%%%%%%%%%%%%%%%%%%%%%%%%%%%%%%%%%%%%%%%%%%%%%%%%%%%
\section{Background}
%%%%%%%%%%%%%%%%%%%%%%%%%%%%%%%%%%%%%%%%%%%%%%%%%%%%%%%%%%%
%%%%%%%%%%%%%%%%%%%%%%%%%%%%%%%%%%%%%%%%%%%%%%%%%%%%%%%%%%%
\subsection{Local cohomology and Lyubeznik numbers}\label{Sec LC-LyuNum}
%%%%%%%%%%%%%%%%%%%%%%%%%%%%%%%%%%%%%%%%%%%%%%%%%%%%%%%%%%%
In this section we recall definitions and properties of local cohomology and Lyubeznik numbers that we discuss in further sections. We refer the interested reader to \cite{BroSharp,TwentyFourHours} for local cohomology and \cite{SurveyLyuNum} for Lyubeznik numbers.

Let $R$ be a Noetherian ring, $I \subset R$ an ideal, and  $M$ an $R$-module. Suppose that $I$ is generated by $\underline{f}=f_1,\ldots, f_\ell.$
We consider the $\Cech$ech complex:
$$
C^{\bullet}(f;M): \-\- 0\to M \to \oplus_{j} M_{f_j} \to 
\oplus_{i<j} M_{f_i f_j}\to \ldots 
\to M_{f_1 \cdots f_\ell} \to 0.
$$

The $j$-th cohomology of this complex, $H^j_I(M),$ is called the \emph{local cohomology module} supported at the ideal $I=(f_1,\cdots,f_\ell)$. We point out that $H^j_I(M)$ does not depend on the choice of generators for $I.$ We note that local cohomology commutes with flat extensions; in particular, $\left(H^i_I(R)\right)_{\fp}=H^i_I(R_{\fp})$ for every prime ideal ${\fp}\subset R.$

\begin{definition}
Let $M$ be an $R$-module and ${\fp}$ be a prime ideal.
We define  \emph{the $i$-th Bass number of $M$ with respect to $\fp$}, 
denoted $\mu^i(\fp, M),$ as $\dim_{R_\fp/\fp R_\fp} \Ext^i_R(R_\fp/\fp R_\fp, M_\fp).$
\end{definition}

Suppose that $M$ is an $R$-module.
Let $0\to M\to E^0\to E^1\to \ldots $ be a minimal injective resolution for $M.$ 
The number of copies of  the injective hull of $R/\fp,$ $E_R(R/\fp),$ in $E^i$ is given by $\mu^i(\fp,M).$

\begin{theorem}[{see \cite[ Theorem $2.1$]{Huneke} and \cite[Theorem $3.4$]{LyuDMod}}]
Let $S$ be a regular ring containing a field.
Every local cohomology module $H^i_I(S)$ has finite Bass numbers.
That is, for every prime ideal ${\fp}\subset S$,
$\dim\Ext^{i}_{S_{\fp}}(S_{\fp}/{\fp} S_{\fp},H^j_I(S_{\fp}))$ is finite.
\end{theorem}

\begin{definition}[{see \cite[Theorem-Definition $4.1$]{LyuDMod}}]
Let $(R,\m, K)$ be a local ring containing a field. By the Cohen Structure Theorems, there exists a surjective homomorphism 
$\pi:S \to \widehat{R}$, where
$\widehat{R}$ is the completion of $R$, and 
$S=K[[x_1,\ldots,x_n]]$ for some $n.$
If $I=\Ker(\pi)$, 
the \textit{Lyubeznik number of $R$ with respect to integers $i, j \geq 0$} is defined as
$$
\lambda_{i,j}(R)=\mu^i(\m,H_I^{n-j}(S))=\dim_K \Ext_S^i(K,H_I^{n-j}(S)),
$$ 
which is finite and only depends on the ring $R$ and the integers $i$ and $j$.
If $d=\dim(R),$ $\lambda_{d,d}(R)$ is called the 
\emph{highest Lyubeznik number of }$R$ (this is justified because $\lambda_{i,j}(R)=0$ if $i>d$ or $j>d$). 
We can arrange the Lyubeznik number in a matrix $(\lambda_{i,j}(R)).$ We say that $R$ has a {\it trivial Lyubeznik table} if $\lambda_{d,d}(R)=1$ and all the other vanish.
\end{definition}

\begin{definition}[\cite{HHgraph}]
Let $R$ be a local ring. The \emph{Hochster-Huneke graph $\Gamma_R$ of $R$} is defined as follows. Its vertices are the minimal prime ideals, $\fp,$ of $R$ such that 
$\dim(R)=\dim(R/\fp).$ Two different
vertices $P$ and $\mathfrak{\fq}$ are joined by an edge if and only if $\hbox{ht}_R(\fp+\fq)=1$. 
\end{definition}  

\begin{theorem}[{see \cite[Main Theorem]{W}}]\label{ThmHighestLyuNum}
Let $R$ be a $d$-dimensional local ring that contains a field. Then $\lambda_{d,d}(R)$ is equal to the number of connected components of the Hochster-Huneke graph of $B$, where $B=\widehat{R^{sh}}$ is the completion of the strict Henzelization of $R$.
\end{theorem}

\begin{remark}\label{Rem HH Mon}
The set of associated primes of a squarefree monomial ideal depends only on the minimal monomial generators (cf. \cite[Proposition 1.2.2 and Corollary 1.3.10]{HerzogHibi}). Therefore, the Hochster-Huneke graph for a Stanley-Reisner ring is independent of the base field.
\end{remark}

%%%%%%%%%%%%%%%%%%%%%%%%%%%%%%%%%%%%%%%%%%%%%%%%%%%%%%%%%%%
\subsection{Monomial ideals and polarization}\label{Pre Monomials}
%%%%%%%%%%%%%%%%%%%%%%%%%%%%%%%%%%%%%%%%%%%%%%%%%%%%%%%%%%%
In Sections \ref{SecLocSR} and \ref{SecPol}, we study the behavior of Lyubeznik numbers of Stanley-Reisner rings. For more detail about local cohomology with support in monomial ideals, we refer to \cite{AGZ,AM-Proc,AMV,Mustata,Terai,YSq,YStr}. We also refer to \cite{Peeva,SaraPol,Y12} for details about polarization.

\begin{definition}
Let $S=K[x_1,\ldots,x_n].$
Let $M= \bigoplus_{\ba \in \bZ^n} M_{\ba}$ be finitely generated $\bZ^n$-graded $S$-module. Given $\beta\in \ZZ^n,$ we take the shifted module $M(-\alpha)$
by $\left(M(\beta)\right)_{\alpha}=M_{\alpha+\beta}.$
\end{definition}

We recall the category of squerefree modules.
\begin{definition}[\cite{YSq}] 
Let $S=K[x_1,\ldots,x_n].$
We say a finitely generated $\bN^n$-graded $S$-module $M= \bigoplus_{\ba \in \bN^n} M_{\ba}$ is {\it squarefree}, if
the multiplication maps $M_{\ba} \ni y \longmapsto x_iy \in M_{\ba +\be_i}$ is bijective for all
$\ba \in \bN^n$ and all $i \in \supp (\ba) := \{ i \mid \alpha_i \ne 0 \}$.
\end{definition}

Here we list some basic properties of these graded modules.

\begin{itemize}
\item For a monomial ideal $I$, it is a squarefree $S$-module, if and only if $S/I$ is a squarefree module, 
  if and only if  $I =\sqrt{I}$.
The free modules $S$ itself and the $\bZ^n$-graded canonical module $\omega_S=S(-{\mathbf 1})$ are squarefree.
Here ${\mathbf 1} = (1,1, \ldots, 1) \in \bN^n$.

\medskip

\item Let $\gmod S$ be the category of $\bZ^n$-graded finitely generated $S$-modules, and
$\Sq S$ its full subcategory consisting of squarefree modules. Then $\Sq S$ is an Abelian subcategory of $\gmod S$.
\end{itemize}

For each $\ba \in \bN^n$, E. Miller \cite{Mi00} introduced the notion of {\it positively $\ba$-determined} $S$-modules. 
Here we do not gives the precise definition, but just list some basic properties.
\begin{itemize}
\item A monomial ideal $I$ (equivalently $S/I$)  is positively $\ba$-determined 
if and only if $I$ is generated by monomials of the form $x^\bb$ for some $ \bb \preceq \ba$. 

\medskip
 
\item A positively ${\mathbf 1}$-determined $S$-module is nothing other than a squarefree $S$-module. 

\medskip

\item Let $\gmod_\ba S$ be the full subcategory of $\gmod S$ consisting of positively $\ba$-determined $S$-modules. 
Then $\gmod_\ba S$ is  an Abelian subcategory of $\gmod S$. 

\medskip

\item If $M \in \gmod_\ba S$, then $\Ext^i_S(M, S(-\ba)) \in \gmod_\ba S$.
\end{itemize}

We now fix  notation and relate squarefree monomial ideals with simplicial complexes. 
Set $[n]:=\{1, \ldots, n\}$ and $S=K[x_1, \ldots, x_n]$. 
For a simplicial complex $\Delta \subset 2^{[n]}$, consider a squarefree monomial ideal  $I_\Delta := (\prod_{i \in \gs} x_i \mid \gs \subset [n], \gs \not \in \Delta)$ of $S$.  
Any squarefree monomial ideal is of the form $I_\Delta$ for some $\Delta$, and  $K[\Delta]:= S/I_\Delta$ is the 
Stanley-Reisner ring of $\Delta$ over $K$. 
Let $\m=(x_1,\ldots,x_n).$
The third author \cite[Theorem $1.1$]{YStr} showed that the Lyubeznik numbers of $K[\Delta]_\m$ can be obtained from its  graded structure:
\begin{equation}\label{LyuNum Graded}
\lambda_{i,j}(K[\Delta]_{\m})=\dim_K[\Ext_S^{n-i}(\Ext^{n-j}_S(K[\Delta], \omega_S), \omega_S)]_0.
\end{equation}

For $\gs \subset [n]$, let ${\fp}_\gs := (x_i \mid i \not \in \gs )$ be a monomial prime ideal of $S$, and set 
$K[\gs]:= S/{\fp}_\gs \cong K[x_i \mid i \in \gs]$. We remark that $K[\gs]$ can be seen as a subring of $S$ in the natural way.   
Recall that the {\it link} of $\gs$ in a simplicial complex $\Delta$ is defined by  
$$\lk_\Delta \gs := \{ \tau \subset [n] \mid \tau \cap \gs = \emptyset, \tau \cup \gs \in \Delta\}.$$
It is a simplicial complex again, and we have 
$$K[\Delta]_{{\fp}_\gs} \cong L[\lk_\Delta \gs^\rc]_\fn,$$
where $L$ is an extended field of $K$, $\fn$ is the graded maximal ideal of  $L[\gs]=L[x_i \mid i \in \gs]$, 
and we set  $\gs^\rc := [n] \setminus \gs$. 
From the previous discussion and Equation \ref{LyuNum Graded}, we have 

\begin{eqnarray}\label{localizing S-R}
 \lambda_{i,j}(K[\Delta]_{{\fp}_\gs})  &=& \lambda_{i,j}( L[\lk_\Delta \gs^\rc]_\fn)\\
&=& \dim_L \Ext_{L[\gs]}^{\#\gs -i} (\Ext_{L[\gs]}^{\#\gs -j}( L[\lk_\Delta \gs^\rc], \omega_{L[\gs]}), \omega_{L[\gs]}) \nonumber \\
&=&\dim_K \Ext_{K[\gs]}^{\#\gs -i} (\Ext_{K[\gs]}^{\#\gs -j}( K[\lk_\Delta \gs^\rc], \omega_{K[\gs]}), \omega_{K[\gs]}). \nonumber 
\end{eqnarray}

To control the last term of \eqref{localizing S-R}, we have to step into the theory squarefree modules.  
We say $\alpha \in \NN^n$ is {\it squarefree} if $\alpha_i = 0,1$ for all $i$. 
When $\alpha$ is squarefree, we freely identify $\alpha$ with $\supp(\alpha)$. 
For example, if $M$ is a $\ZZ^n$-graded module and $\gs \subset [n]$, 
$M(\gs)$ denotes the shifted module of $M$ by the corresponding 0-1 vector.  
Similarly, $M_\gs$ is the homogeneous component of $M$ with the corresponding  degree. 
If $M$ is a $\ZZ^n$-graded $S$-module, through the ring homomorphism $K[\gs] \hookrightarrow S$, we can regard 
$$M|_\gs := \bigoplus_{\substack{\alpha \in \NN^n \\ \supp(\alpha) \subset \gs}} M_\gs$$
as a $\NN^{\# \gs}$-graded $K[\gs]$-module. 
For a squarefree $S$-module $M$, it is easy to see that $$\fl_\gs(M):= M(\sigma^\rc)|_\gs$$ is a squarefree $K[\gs]$-module with 
$$(\fl_\gs(M))_\tau = M_{\gs^\rc \cup \tau}$$ for each $\tau \subset \gs$.    
Hence we have an exact functor $\fl_\gs: \Sq S \to \Sq K[\gs]$. Straightforward computation shown that 
$$\fl_\gs(K[\Delta]) \cong K[\lk_\Delta \gs^\rc]=K[\gs]/I_{\lk_\Delta \gs^\rc}.$$
In particular, for $\tau \subset [n]$, we have  
$$\fl_\gs(K[\tau]) \cong \begin{cases}
K[\tau \cap \gs] & \text{if $\tau \supset \gs^\rc$,}\\
0 & \text{otherwise.}
\end{cases}
$$
%%%---
 
\begin{definition}
Let $g=x_1 ^{\alpha_1}\cdots x_n^{\alpha_n}$ be a monomial in $S=K[x_1,\dots,x_n]$. We define another monomial $g'$ in the polynomial ring $\widetilde{S}=K[x_{i,j}]$ where 
%$$g'=x_{1,1}\cdots x_{1,\alpha_1}x_{2,1}\cdots x_{2,\alpha_2}x_{n,\alpha_n}.$$ 
$$g'= \prod_{1 \le i \le n} x_{i,1} x_{i,2} \cdots
x_{i, \alpha_i}$$
If $I=(g_1,\ldots,g_k)$ is an ideal in $S$, the ideal $\widetilde{I}=(g_1 ', \ldots ,g_k ')$ in $\widetilde{S}$ is called the {\it polarization} of $I$.
\end{definition}

Polarization is a classical technique constructing  
a squarefree monomial ideal $\wI \subset \wS$ from a general  monomial ideal $I \subset S$.
There are relations between the singularity of a monomial ideal and its polarization. For instance, the Betti numbers of $S/I$ over $S$ and $\widetilde{S}/\widetilde{I}$ over $\widetilde{S}$ are same (cf. \cite[Theorem $21.10$]{Peeva}). As a consequence, $S/I$ is Cohen-Macaulay if and only if $\widetilde{S}/\widetilde{I}$ is Cohen-Macaulay.

In \cite{Y12}, the third author extended the polarization operation $I \mapsto \wI$ to the functor 
$\pol : \gmod_\ba S \to \Sq \wS$. Here $\pol(I) =\wI$ and $\pol(S/I) = \wS/\wI$. 
Moreover,  $$\Theta := \{ x_{i,1}-x_{i,j} \mid 1 \leq i \leq n, \, 2 \leq j \leq a_i \, \} \subset \wR$$ forms a $\pol(M)$-regular sequence, and gives an isomorphism 
$\pol(M) \otimes_{\wS} \wS/(\Theta) \cong M$.\footnote{An essentially same functor had been introduced by 
 Sbarra \cite{Sb09}, but we use the convention of \cite{Y12} here.}   
Here we just remark that $[\pol(M)]_\zero \cong M_\zero$. 
The ``reversed copy'' $\dpol$ of $\pol$ is also a polarization functor $\gmod_\ba S 
\to \Sq \wS$, but use the convention that $\dpol(x^\bb)=\prod 
x_{i,\alpha_i} x_{i,\alpha_i-1} \cdots x_{i, \alpha_i-\beta_i+1} \in \wS$. 

%%%---

We now present a well-known relation between polarization and localization, which will play a major role in the proof of 
Proposition \ref{Prop Pol Gral LyuNum}.

\begin{remark}\label{RelLocPol}
Suppose that $I$ is a monomial ideal in a polynomial ring 
$S = K[x_1,\ldots,x_n]$, and  $\widetilde{I}$ is its polarization in $\widetilde{S} =K[x_{i,j}]$. 
We consider an inclusion $\iota:S\to \widetilde{S}$ given by $x_i \mapsto x_{i,1}$.
Let ${\fp}=(x_{i,1},\ldots, x_{n,1})$ be the prime ideal generated by the image of the variables $x_i$. We have that
$$
\left( \frac{\widetilde{S}}{\widetilde{I}}\right)_{\fp}= 
\left( \frac{\widetilde{S}}{\iota\left(\sqrt{I}\right) \widetilde{S}} \right)_{\fp}.
$$
Moreover, if $L=\hbox{Frac}\left(K[x_{i,j}]_{j>1}\right),$ and $T=L[[x_1,\ldots,x_n]]$
$$
\widehat{\left( \widetilde{S}/\widetilde{I}\right)_{\fp}}\cong
 \frac{T}{\sqrt{I} T} .
$$
\end{remark}

%%%%%%%%%%%%%%%%%%%%%%%%%%%%%%%%%%%%%%%%%%%%%%%%%%%%%%%%%%%
\subsection{$F$-modules}
%%%%%%%%%%%%%%%%%%%%%%%%%%%%%%%%%%%%%%%%%%%%%%%%%%%%%%%%%%%
We recall some definitions and  properties of the theory of $F$-modules introduced by Lyubeznik \cite{LyuFmod}.

A morphism of rings $\varphi:R\to S$ gives a change of base functor, $S\otimes_{R} M,$ from the category of  $R$-modules to $S$-modules.
Suppose that $S$ is a regular ring of prime characteristic $p>0.$
We are interested in the case when $S=R$ and $\varphi$ is the Frobenius morphism, and the functor is denoted by $\cF M$ \cite{P-S}.
For example, if $M$ is the cokernel of a matrix $(a_{i,j}),$ then $\cF M$ is the cokernel of $(a^p_{i,j}).$ 

\begin{definition}[{see \cite[Definition $1.1$]{LyuFmod}}]
Let $S$ be a regular ring of prime characteristic $p>0.$
An $S$-module, $\cM$, is an {\it $F$-module} if there exists an isomorphism of $S$-modules 
$\nu :\cM\to F \cM.$
\end{definition}

If $M$ is an $S$-module and $\alpha :M\to \cF M$ is a morphism of 
$S$-modules, we  consider
$$
\cM=\lim\limits_\to (M\FDer{\alpha} \cF M\FDer{\cF\alpha} \cF^{2} M\FDer{\cF^2 \alpha} \ldots).
$$
We have that $\cM$ is an {\it $F$-module} with  $\cM\FDer{\alpha}\cF\cM$ as structure isomorphism,
if $M$ is a finitely generated $S$-module, we say that $\cM$ is an \emph{$F$-finite $F$-module} with generator $\alpha :M\to \cF M$.
\begin{remark}
Let $S$ be a regular ring of prime characteristic $p>0.$
Let $I$ and $J$ be ideals of $S.$
Then $H^i_I(S)$ and $H^i_IH^j_J(S)$ are $F$-finite $F$-modules.
\end{remark}

\begin{theorem}[{see \cite[Theorem $2.11$]{LyuFmod}}]
Let $S$ be a regular ring of prime characteristic.
Every $F$-finite $F$-module, $M$, has finite Bass numbers. That is, for every prime ideal ${\fp}\subset S$,
$\dim\Ext^{i}_{S_{\fp}}(S_{\fp} /{\fp}S_{\fp},M_{\fp})$ is finite.
\end{theorem}

%%%%%%%%%%%%%%%%%%%%%%%%%%%%%%%%%%%%%%%%%%%%%%%%%%%%%%%%%%%
\subsection{$D$-modules} 
In this section, we recall some relevant facts about $D$-modules.  We refer the interested reader to \cite{Bj1,Bj2,Cou}.

\begin{definition}
Suppose that $S$ is either $K[x_1,\ldots,x_n]$ or $K[[x_1,\ldots,x_n]]$.  We define the rings of $K$-linear differential operators of $S$ by 
$$D_K(S)=S\left\langle\frac{1}{t!} \frac{\partial}{\partial x_i} \ | \ t \in \NN, 1 \leq i \leq n \right\rangle\subseteq \Hom_K (S,S).$$ 
\end{definition}

We have that  $S_f$ is an $D_K(S)$-module of finite length  for every $f \in S$. As a consequence, the local cohomology modules $H^j_I(S)$ and $H^i_\m H^j_I(S)$ have finite length as $D_K(S)$-modules \cite[Corollary 6]{Lyu2}.

\begin{definition}[{\cite[Definition $4.3$]{NuWi1}}]\label{DefGenLyuNum}
Let $(R, \m, K)$ be a local ring containing a field  and let $\widehat{R}$ be its completion at $\m$.   
For $L$ a coefficient field of $\widehat{R}$, by Cohen-Macaulay Structure Theorems there exists a surjection $\pi: S \to \widehat{R}$, where $S = K[[x_1, \ldots, x_n]]$ for some $n \geq 1$, and such that $\pi(K) = L$.
The \emph{$i$-th generalized Lyubeznik number of $R$ with respect to $L$} is defined as  
 $$\lambda^{i}_{0}(R;L):=\Length_{D_K(S)}  H^{n-i}_{I}(S).$$
This number is finite and depends only on $R$ and $L$.
\end{definition}

In the previous definition, we considered very specific generalized 
Lyubeznik numbers. The definition is more general and includes the original Lyubeznik numbers.

\begin{remark}
If in the definition above, $R$ is a localization at the homogeneous maximal ideal of an Stanley-Reisner ring, we have  the generalized Lyubeznik numbers could depend on the characteristic of the coefficient field, but not on the choice of a specific field. This is because one can relate the length of $H^{n-i}_I(S)$ with length as a straight module (cf. \cite[Proposition $2.10$ and Remark $2.12$]{YStr} and \cite[Theorem $6.6$]{NuWi1}). 
\end{remark}

\begin{remark}\label{RemGamma}
If $S=K[[x_1,\ldots,x_n]]$ is a power series and $I$ is a monomial ideal, then there exists a filtration 
$$
0=M_1\subset M_2\subset\ldots\subset M_\ell=H^i_I(S)
$$
such that $M_{j+1}/M_j\cong H^{|\alpha|}_{\fp_\alpha}(S),$
where $\alpha=(\alpha_1,\ldots,\alpha_n)\in \{ 0,1\}^n$, $
| \alpha | =
\alpha_1+\ldots+\alpha_n,$ and $\fp_\alpha=(x_{\alpha_k}|\alpha_k\neq 0).$
Since $H^{|\alpha|}_{\fp_\alpha}(S)$ is a simple $D_K(S)$-modules, we have that 
$\Length_{D_K(S)} H^i_I(S)=\ell.$
\end{remark}

\begin{definition}[{see \cite[Section 2]{AGZ} \cite[Section $3$]{AM-Proc}}]\label{AMProc}
Under the notation of Remark \ref{RemGamma}, 
we take the multiplicities
$$
m_{i,-\alpha}=\#\{j|M_{j+1}/M_j\cong H^{|\alpha|}_{\fp_\alpha}(S)\}.
$$

We define the numbers
$$\gamma_{i,j}=\sum_{|\alpha|=n-i}m_{j,-\alpha},$$
which are invariants depending only on $R=S/I$.
\end{definition}

\begin{remark}
Under the notation in Remark \ref{RemGamma} and assuming that $I$ is squarefree,
we note that in \cite[Proposition 2.1]{AGZ} it is shown that 
$$
m_{i,\alpha}=\dim_K\left(H^i_{I}(S)\right)_{\alpha},
$$
where $\left(H^i_I(S)\right)_\alpha$ denotes the $K$-vector space of degree $\alpha$
in the $\ZZ^n$-graded module $H^i_I(S).$ 
\end{remark}

\begin{remark}\label{RemMultBetti}
Under the notation in Remark \ref{RemGamma} and assuming that $I$ is squarefree,
we note that in \cite[Corollary 2.2]{AGZ} it is shown that 
$$
m_{i-|\alpha|,\alpha}=\beta_{i,-\alpha}(I^\vee)
$$
where $\beta_{i,-\alpha}(I^\vee)$ denotes the $\alpha$-Betti number of the  $\ZZ^n$-graded minimal free resolution of the Alexander dual of $I$   (see also \cite[Theorem 3.3]{Mustata}).
As a consequence,
$$
\gamma_{i,j}=\sum_{|\alpha|=n-i}m_{j,\alpha}=
\sum_{|\alpha|=n-i} \beta_{j+|\alpha|,-\alpha}(I^\vee),
$$
which is the $(i+n-i)$-Betti number of $I^\vee.$
In addition,
$$
\lambda^0_{i}(S/I)=\sum_{\alpha} m_{n-i,\alpha}=\sum_{\alpha}
\beta_{n-i+|\alpha|,-\alpha}(I^\vee),
$$
which is the total Betti number of the $(n-i)$-linear strand of the free resolution of $I^\vee.$
\end{remark}

%%%%%%%%%%%%%%%%%%%%%%%%%%%%%%%%%%%%%%%%%%%%%%%%%%%%%%%%%%%
\section{A global bound  for Lyubeznik numbers}
%%%%%%%%%%%%%%%%%%%%%%%%%%%%%%%%%%%%%%%%%%%%%%%%%%%%%%%%%%%
In this section we prove first main theorem using $F$-modules 
theory.
\begin{remark}\label{RemDuality}
Suppose that $(S,\m,K)$ is a regular local ring.
Let $n=\dim(S).$
Let $M$ be a finitely generated $S$-module.
Let $\underline{x}=x_1,\ldots,x_n\in \m$ be a regular system of parameters, and
let $\cK(\underline{x},-)$ denote the associated Koszul complex.
We have that
$$
\Tor^S_i(K,M)\cong H_i(\cK(\underline{x},M))\cong H^{n-i}( \cK(\underline{x},M))=\Ext^{n-i}_S(K,M).
$$
Then,
$$
\dim_K\Tor^S_i(K,M)=\dim_K\Ext^{n-i}_S(K,M).
$$
\end{remark}
Let $S$ denote a regular Noetherian ring, and $M$ be a finitely generated $S$ module.
Let ${\fp}$ be a prime ideal of $S$ and $d_{\fp}$ denote $\dim(S_{\fp}).$
Let $\beta_i({\fp},M)=\dim_{S_{\fp}/{\fp} S_{\fp}}\Tor^{S_{\fp}}_i(S_{\fp}/{\fp} S_{\fp},M_{\fp}).$
Remark \ref{RemDuality} has two implications.
First, there exists a bound, $B,$ given by the maximum rank of the modules appearing in a free resolution for $M$ such that
\begin{equation}\label{BoundFG}
\mu^i({\fp},M)=\beta_{d_{\fp}-i}(S_{\fp},M_{\fp})\leq B.
\end{equation}
Second, if ${\fp}'$ is another prime ideal contained in ${\fp}$, we have that 
\begin{equation}\label{DecreaseFG}
\mu^i({\fp}',M)=\beta_{d_{{\fp}'}-i}({\fp}',M)\leq \beta_{d_{{\fp}'}-i}({\fp},M)=
\mu^{d_{\fp}-d_{{\fp}'}+i}({\fp},M).
\end{equation}

We need to recall a basic property of $F$-finite $F$-modules.

\begin{lemma}[{see \cite[Remark $2.13$]{LyuFmod}}]
\label{LemmaBass}
Suppose that $(S,\m,K)$ is a regular local ring of positive characteristic.
Let $\cM$ be an $F$-module with a generator
 $M\FDer{\alpha}\cF M.$ Then,
$$
\mu^i(K,\cM)\leq \mu^i(K,M).
$$
\end{lemma}

\begin{theorem}\label{Main Localization}
Let $S$ be a regular Noetherian ring  of prime characteristic $p$ and dimension $n.$ Let $\cM$ be an $F$-finite $F$-module.
Then, there exists an integer $B$ such that
$$
\mu^i({\fp},\cM)\leq B
$$ 
for every ${\fp}\in\Spec(S),$ and  $i\in\NN$.
In particular, for every ideal $I\subset S,$ the Lyubeznik numbers of $R=S/I$, $\lambda_{i,j}(R_{\fp})=\mu^i({\fp},H^{n-j}_I(S)),$ are bounded.
\end{theorem}

\begin{proof}
Let $n=\dim S.$
Let $M\to \cF M$ be a generator for $\cM$ such that $M$ is a finitely generated $R$-module.
There exists  a free resolution for $M$
$$
G_{\bullet}:\ldots\to G_{i}\to \ldots \to G_0\to M\to 0
$$
such that
$b_i=\rk(G_i)$ is finite.
Let 
$$B=\Max\{b_i\mid i=1,\ldots, n\}.$$
By the characterization of Betti numbers as rank of the terms 
in a minimal free resolution, we have that
$$\dim_{S_{\fp}/{\fp} S_{\fp}}(\Tor^{S_{\fp}}_i(S_{\fp}/\fp S_{\fp},M_{\fp}))\leq b_i$$
for every $i\in\NN,$ and prime ideal ${\fp}\subset S.$
We have that 
$M_{\fp}\to \cF_{S_{\fp}} M_{\fp}\cong (\cF M)_{\fp}$ is a root for the $F_{S_{\fp}}$-module $\cM_{\fp}$ 
for every prime ideal ${\fp}\subset S,$
\cite[Definition-Proposition $1.3$]{LyuFmod}.
Then,
\begin{align*}
\mu^i({\fp},\cM) &= \mu^i({\fp},M)\hbox{ by Lemma \ref{LemmaBass}.}\\
&=\dim_{S_{\fp}/{\fp}S_{\fp}}\Ext^i_{S_{\fp}}(S_{\fp}/{\fp} S_{\fp},M_{\fp})\\
&=\dim_{S_{\fp}/{\fp}S_{\fp}}\Tor^{S_{\fp}}_{\dim S_{\fp} -i}(S_{\fp}/{\fp}S_{\fp},M_{\fp})\hbox{ by Remark \ref{RemDuality}.}\\
& \leq \rk(G_{\dim S_{\fp}-i})\hbox{ because }G_\bullet\otimes S_{\fp}\hbox{ is a resolution for }M_{\fp}.\\
&= b_{\dim S_{\fp}-i}\\
&\leq B
\end{align*}

Hence, $\mu^i({\fp},\cM) \leq B,$
which proves the claim for $F$-finite $F$-modules.
The claim about Lyubeznik numbers and local cohomology follows from the fact that $H^{n-i}_I(S)$ is always an $F$-finite $F$-module.
\end{proof}

\begin{remark}
Under the notation of Theorem \ref{Main Localization}, we can take 
\begin{equation}\label{Bound}
B=\Max\{\dim_{S/\m}\Tor^S_i(S/\m,\Ext^j_S(S/I,S)) | \m\hbox{ is a maximal ideal of }S\hbox{ and }i,j\in\NN\}
\end{equation}
as a bound of the Lyubeznik numbers of $S/I.$
Then, one can bound the Lyubeznik numbers of $S/I$ in terms of Betti numbers of its degeneracy modules, $\Ext^j_S(S/I,S).$ 
Puthenpurakal's bound for polynomial rings of characteristic zero is given by the multiplicity of  $H^i_I(S)$ as holonomic $D_K(S)$-module \cite{TonyInjRes}. The bound presented in Theorem \ref{Main Localization} could be sharper.
For example, if $f$ is a polynomial, we have that \ref{Bound}
gives  $B=1$ as a bound, which is sharp. The bound presented in \cite{TonyInjRes} is bigger than or equal to
$(\deg(f))^n.$ Unlike the bound presented in \cite{TonyInjRes}, $B$ is shown to be a bound for the localization at every prime ideal.  
\end{remark}

%%%%%%%%%%%%%%%%%%%%%%%%%%%%%%%%%%%%%%%%%%%%%%%%%%%%%%%%%%%
\section{Localization and Lyubeznik numbers of Stanley-Reisner rings}
\label{SecLocSR}
%%%%%%%%%%%%%%%%%%%%%%%%%%%%%%%%%%%%%%%%%%%%%%%%%%%%%%%%%%%
In this section, we study in deeper detail  the Lyubeznik numbers of an Stanley-Reisner ring. 
In particular, we discuss a lower bound of $B$ presented in Theorem \ref{Main Localization}, under the assumption that $R$ is the localization of a Stanley-Reisner ring at its graded maximal ideal (see Theorem \ref{Thm LyuNum Loc} for details).

Huneke and Sharp \cite{Huneke} and Lyubeznik \cite{LyuDMod} proved 
that the Bass numbers of local cohomology are finite.
Motivated by this result and Theorem \ref{Main Localization}, one could expect that the Lyubeznik numbers, which are Bass numbers of local cohomology modules, would also resemble the behavior of Bass numbers of a finitely generated modules under localization (see \ref{DecreaseFG}). This is not true, even for the highest Lyubeznik number. The following example shows that the highest Lyubeznik number can both decrease and increase under localization.
\begin{example}\label{ExampleBadB}
Let $S=K[x,y,z,u,v],$ and $I=(x,y)\cap (y,z)\cap (z,u)\cap (u,v).$
We have that $\dim R=4,$ and the Hochster-Huneke graph of $R$ is connected.
%\begin{figure}
%\begin{tikzpicture}
%\coordinate [label=center:{$\boldsymbol{G:}$}] (0) at (-1,1);
%\coordinate [label=center:{$(x,y)$}] (0) at (0,-.3);
%\coordinate [label=center:{$(y,z)$}] (1) at (2,2.3);
%\coordinate [label=center:{$(z,u)$}] (2) at (4,2.3);
%\coordinate [label=center:{$(u,v)$}] (3) at (6,-.3);
%\draw (0,0) -- (2,2);
%\draw (2,2) -- (4,2);
%\draw (4,2) -- (6,0);
%\draw [fill=black] (0,0) circle [radius=.05];
%\draw [fill=black] (2,2) circle [radius=.05];
%\draw [fill=black] (4,2) circle [radius=.05];
%\draw [fill=black] (6,0) circle [radius=.05];
%\end{tikzpicture}
%\end{figure}
%\begin{center}
%\begin{tikzpicture}
%[scale=.8, vertices/.style={draw, fill=black, circle, inner sep=0.5pt}]
%\useasboundingbox (-2.5,-2.4) rectangle (2.5,2.4);
%\node [anchor=base] at (-4,1){$G'$:};
%\node [vertices] (1) at (0:2) {};
%\node [anchor=base] at (2.4,-.5) {$(x,y)$};
%\node [vertices] (2) at (60:2) {};
%\node [anchor=base] at (60:2.3) {$(y,z)$};
%\node [vertices] (3) at (120:2) {};
%\node [anchor=base] at (120:2.3) {$(z,u)$};
%\node [vertices] (4) at (180:2) {};
%\node [anchor=base] at (-2.4, -.5) {$(u,v)$};
%\foreach \to/\from in {2/1, 3/2, 3/4}
%\draw [-] (\to)--(\from);
%\end{tikzpicture}
%\end{center}
Then, we have that the highest Lyubeznik number, $\lambda_{4,4}(R),$ is 
$1.$

Let ${\fp}=(x,y,u,v).$
We have that the completion of $R_{\fp}$ is $L[[x,y,u,v]]/J,$ where $J=(x,y)\cap (u,v),$ where $L=R_{\fp}/{\fp}R_{\fp}.$
Then $\dim(R_{\fp})=3$ and the Hochster-Huneke of $R_{\fp}$ consists of two disconnected points.

Then, the highest Lyubeznik number of $R_{\fp}$,  $\lambda_{3,3}(R_{\fp})$, is $2.$ 
If $\fp=(x,y),$ we have that its highest Lyubeznik number of $R_\fp$ is $1$.
\end{example}
In Remark \ref{Grow with Loc}, we present a general technique to build rings whose localization can have very different Lyubeznik numbers.

\begin{remark}
\`Alvarez-Montaner and Vahidi have previously studied the behavior of Bass numbers of local cohomology over squarefree monomial ideals \cite[Section $5$]{AMV}. In particular, they studied the vanishing of these numbers. In addition, they presented examples that show different behavior of the Bass numbers of local cohomology modules and finitely generated modules.
\end{remark}

\begin{lemma}\label{fl and Ext}
Let $S=K[x_1,\ldots,x_n]$ be a polynomial ring over a field.
For a squarefree $S$-module $M$, we have 
$$\fl_\gs (\Ext_S^i(M, \omega_S)) \cong \Ext^i_{K[\gs]}(\fl_\gs(M), \omega_{K[\gs]}).$$
\end{lemma}

\begin{proof}
In \cite[\S3]{Yan04}, the third author constructed the cochain complex 
$$\bD_S(M): 0 \longrightarrow  \bD^{-n}_S(M)  \longrightarrow \bD^{-n+1}_S(M)  \longrightarrow  \cdots  \longrightarrow \bD^0_S(M)  \longrightarrow 0$$
with 
$$\bD^{-i}_S(M) = \bigoplus_{\substack{\gs \subset [n] \\ \# \gs =i}}   (M_\gs)^* \otimes_K K[\gs]$$
such that $H^{-i}(\bD_S(M)) \cong \Ext_S^{n-i}(M, \omega_S)$.  
Here $(-)^*$ means the $K$-dual of a vector space. 

We have 
\begin{eqnarray*}
\fl_\gs(\bD_S^{-i}(M)) &\cong& \fl_\gs   \left( \bigoplus\nolimits_{\substack{\tau \subset [n] \\ \# \tau =i}} 
(M_\tau)^* \otimes_K K[\tau] \right)\\ 
&\cong& \bigoplus_{\substack{\tau \subset [n] \\ \# \tau =i}}   (M_\tau)^* \otimes_K \fl_\gs(K[\tau]) \\
&\cong& \bigoplus_{\substack{\gs^\rc \subset \tau \subset [n] \\ \# \tau =i}}   (M_\tau)^* \otimes_K K[\tau \cap \gs] \\
&\cong& \bigoplus_{\substack{\tau \subset \gs \\ \# \tau =i-\gs^\rc}}   (\fl_\gs(M)_\tau)^* \otimes_K K[\tau] \\
&\cong& \bD_{K[\gs]}^{-i+\# \gs^\rc}(\fl_\gs(M)).
\end{eqnarray*}
These isomorphisms commute with the differential maps, and we have an isomorphism 
$$\fl_\gs(\bD_S^\bullet(M) \cong \bD_{K[\gs]}^{\bullet+\# \gs^\rc }(\fl_\gs(M))$$
of cochain complexes.  Hence the assertion follows from the following computation 
\begin{eqnarray*}
\fl_\gs(\Ext^{n-i}_S(M, \omega_S)) 
&\cong& \fl_\gs (H^{-i}(\bD^\bullet_S(M)))\\
&\cong& H^{-i} (\fl_\gs(\bD^\bullet_S(M)))\\
&\cong& H^{-i+\# \gs^\rc}(\bD_{K[\gs]}^\bullet(\fl_\gs(M)))\\
&\cong& \Ext_{K[\gs]}^{\# \gs -i+\# \gs^\rc}(\fl_\gs(M), \omega_{K[\gs]})\\
&\cong& \Ext_{K[\gs]}^{n-i}(\fl_\gs(M), \omega_\gs).
\end{eqnarray*}
\end{proof}

The following is a generalization of Equation \ref{LyuNum Graded}.

\begin{theorem}\label{Thm LyuNum Loc}
Let $S=K[x_1,\ldots,x_n]$ be a polynomial ring over a field.
Let $\Delta\subset 2^{[n]}$ be a simplicial complex.
$$\lambda_{i,j}(K[\Delta]_{{\fp}_\gs}) = \dim_K [\Ext^{\# \gs-i}_S(\Ext_S^{\# \gs -j}(K[\Delta], \omega_S), \omega_S)]_{\gs^\rc}.$$
\end{theorem}

\begin{proof}
By Lemma~\ref{fl and Ext}, we ave
\begin{eqnarray*}
 \fl_\gs(\Ext^i_S(\Ext_S^j(K[\Delta], \omega_S), \omega_S)) 
&\cong& \Ext^i_{K[\gs]}(\fl_\gs(\Ext_S^j(K[\Delta], \omega_S)), \omega_{K[\gs]}) \\
&\cong&  \Ext^i_{K[\gs]}(\Ext_{K[\gs]}^j(\fl_\gs(K[\Delta]), \omega_{K[\gs]}), \omega_{K[\gs]})  \\
&\cong&  \Ext^i_{K[\gs]}(\Ext_{K[\gs]}^j(K[\lk_\Delta \gs^\rc], \omega_{K[\gs]}), \omega_{K[\gs]}). 
\end{eqnarray*}
Hence we have 
\[
[\Ext^i_S(\Ext_S^j(K[\Delta], \omega_S), \omega_S)]_{\gs^\rc} \cong 
 [\Ext^i_{K[\gs]}(\Ext_{K[\gs]}^j(K[\lk_\Delta \gs^\rc], \omega_{K[\gs]}), \omega_{K[\gs]}) ]_0.
\]
Combining this fact with \eqref{localizing S-R}, we are done. 
\end{proof}

The next result easily follows the above theorem.  

\begin{corollary}\label{lower bound}
Let $S=K[x_1,\ldots,x_n]$ be a polynomial ring over a field, and $\m=(x_1,\ldots,x_n)$.
Let $\Delta\subset 2^{[n]}$ be a simplicial complex.
If $R$ is the localization $K[\Delta]_\m$ of a Stanley-Reisner ring $K[\Delta]=S/I_\Delta$, the number $B$ presented in Theorem \ref{Main Localization} must satisfy   
$$B \ge  \dim_K [\Ext^i_S(\Ext_S^j(K[\Delta], \omega_S), \omega_S)]_\gs$$ 
for all $\gs \subset [n]$ and all  $i, j$. 
\end{corollary}

\begin{remark}\label{Grow with Loc}
(1) It is easy to connect Corollary~\ref{lower bound} with the argument in \S3. 
In fact, if $\cdots \to S^{b_i} \to \cdots \to S^{b_0} \to 0$ is a free resolution of $\Ext^j_S(K[\Delta], \omega_S)$, then we have  
$$b_i \ge  \dim_K [\Ext^i_S(\Ext_S^j(K[\Delta], \omega_S), \omega_S)]_\gs.$$ 

(2) Let $I_\Delta \subset S=K[x_1, \ldots, x_n]$ be a Stanley-Reisner ideal, and set $S':= S[x_{n+1}]=K[x_1, \ldots, x_{n+1}]$. 
Consider the monomial ideal $I' := I_\Delta S' \cap (x_{n+1})$
of $S'$. Since the codimension of $I'$ is 1,  $S'/I'$ (more precisely, its localization at the graded maximal ideal of) 
has the trivial Lyubeznik table as shown in 
\cite{AY-14}. However, for a prime ideal ${\fp} := (x_1, \ldots, x_n)$ of $S'$, we have 
$$\lambda_{i,j}((S'/I')_{\fp})= \lambda_{i,j}(S/I_\Delta).$$
Hence the Lyubeznik table of $(S'/I')_{\fp}$ can be far from trivial. 
\end{remark}

Besides the bad behavior under localization of the Lyubeznik numbers, the generalized Lyubeznik numbers and the multiplicities of local cohomology behave as expected with respect to localization.
\begin{proposition}\label{Prop Loc Gral LyuNum}
Let $S=K[x_1,\ldots,x_n]$ be a polynomial ring over a field.
Let $I\subset S$ denote an squarefree monomial ideal and $d=\dim(S/I)$ and $\m=(x_1,\ldots,x_n)$. Let ${\fp}\subset S$ denote any prime ideal containing $I,$  and $h=\dim(S/I)-\dim(S_{\fp}/IS_{\fp})$.
Then,
$$
\lambda^{j-h}_{0}\left(S_{\fp}/IS_{\fp}\right)
\leq \lambda^{j}_0\left(S/I\right)
$$
for $j\in\NN.$ Moreover,
$$
\gamma_{i-h,j-h}\left(S_{\fp}/IS_{\fp}\right)
\leq \gamma_{i,j}\left(S/I\right).
$$
\end{proposition}
\begin{proof}
We fix $j\in\NN.$
Since $I$ is a monomial ideal, there exists a filtration of
$D_K(S)$-modules
$$
0=M_0\subset M_1\subset M_2\subset \ldots \subset M_\ell=H^{n-j}_I(S)
$$
such that $M_{t}/M_{t-1}\cong H^{|\alpha|}_{\fp_{\alpha}}(S)$
for some $\alpha\subset [n]$ 
\cite[Proposition $2.10$ and Remark $2.12$]{YStr}. We recall that
$$m_{j,-\alpha}=\#
\{t|M_{t}/M_{t-1}\cong H^{|\alpha|}_{\fp_{\alpha}}(S)\}.$$
We note that $H^{|\alpha|}_{\fp_\alpha}(R)$ is a simple $D_K(R)$-module, and so,
$\Length_{D_K (S)}H^{n-j}_{I}(S)=\sum_{\alpha} m_{n-j,-\alpha}.$

Let $A=\widehat{S_{\fp}}$ and $L=S_{\fp}/{\fp}S_{\fp}.$
By Cohen Structure Theorems, $A$ is a power series ring over $L.$ 
We have that for every variable $x_r\in S,$  $x_r$ is either a unit or a regular element in $A.$ Hence, $IA$ is still induced from a squarefree monomial ideal. 
We note that $A$ is a flat extension of $S$ and that
$H^{n-j}_I(S)\otimes_S A\cong H^{n-j}_{I}(A)$ as $D_L(A)$-modules. 
In addition, 
$$
0=M_0\otimes_S A\subset M_1\otimes_S A\subset M_2\otimes_S A\subset \ldots \subset M_\ell\otimes_A =H^{n-j}_I(A)
$$
gives a filtration of $D_L(A)$-modules such that
$M_t\otimes_S A/M_{t-1}\otimes_S A\cong H^{|\alpha|}_{\fp_\alpha A}(A)$ is either a simple $D_L(A)$-module or zero (depending on whether $\fp_\alpha\subset {\fp}$).
Let 
$$
m'_{-\alpha}=\#
\{t|M_{t}/M_{t-1}\otimes_S A\cong H^{|\alpha|}_{\fp_{\alpha}}(A)\hbox{ and } \fp_\alpha \subset {\fp}\}.
$$
We note that $m'_{-\alpha}\leq m_{n-j,-\alpha}.$
Hence, 
\begin{align*}
\lambda^{j-h}_{0}(S_{\fp}/IS_{\fp})&= \Length_{D_L (S_{\fp})}H^{\dim(S_{\fp})-j+h}_{I}(S_{\fp})\\
&=\Length_{D_L (S_{\fp})}H^{n-j}_{I}(S_{\fp})\\
&= \sum_{\alpha} m'_{-\alpha}\\
& \leq  \sum_{\alpha} m_{n-j,-\alpha}\\
& =\Length_{D_K (S)}H^{\dim(S)-h-j}_{I}(S)\\
& =\lambda^{j+h}_{0}(S_\m/IS_\m).
\end{align*}
Moreover, 
$$
\gamma_{i-h,j-h}(S_{\fp}/IS_{\fp}) =\sum_{|\alpha|=n-i} n_\alpha 
\leq \sum_{|\alpha|=n-i} m_{n-i,\alpha} =
\gamma_{i,j}(S/I).
$$
\end{proof}

\begin{remark}
As noted in Remark \ref{RemMultBetti}, the length $H^i_{I_\Delta}(R)$ as a $D_K(S)$ is determined by 
the $\ZZ^n$-graded Betti numbers of the Alexander dual of $I_\Delta$. The same applies to the invariants $\gamma_{i,j}$.
We point out that from this characterization, it does not follow that the generalized Lyubeznik number cannot increase under localization. This is because the multigrading is lost after localization; for instance, the linear strands of a graded free resolution are not preserved. 
\end{remark}

%%%%%%%%%%%%%%%%%%%%%%%%%%%%%%%%%%%%%%%%%%%%%%%%%%%%%%%%%%%
\section{Polarization and Lyubeznik numbers of Stanley-Reisner rings}\label{SecPol}
%%%%%%%%%%%%%%%%%%%%%%%%%%%%%%%%%%%%%%%%%%%%%%%%%%%%%%%%%%%
The main theorem in this section we prove that the (original) Lyubeznik numbers
given by a monomial ideal and its polarization are essentially the same.

\begin{theorem}\label{Main Thm Pol LyuNum}
Let $S=K[x_1,\ldots,x_n]$ be a polynomial ring, $\m=(x_1,\ldots,x_n)$ and $I\subset S$ be a monomial ring. Let $\widetilde{I}$ denote the polarization of $I,$ and  $\widetilde{S}=K[x_{r,s}]$ denote the polarization ring. Let $\fn$ be the maximal homogeneous ideal of $\widetilde{S}$
and  $h=\dim(\widetilde{S}/\widetilde{I})-\dim(S/I)$.
Then,
$$
\lambda_{i-h,j-h}\left(S_\m/IS_\m\right)= \lambda_{i,j}\left(\widetilde{S}_\fn/\widetilde{I}\widetilde{S}_\fn\right).
$$
for every $i,j\in\NN.$
\end{theorem}
\begin{proof}
It suffices to prove that 
$$
\left(\Ext^i_S(\Ext^j_S(S/\sqrt{I}, \omega_S),\omega_S)\right)_{\zero} \cong \left(\Ext^i_{\wS}(\Ext^j_{\wS}(\wS/\wI, \omega_{\wS}), \omega_{\wS})\right)_{\zero}
$$
for all $i,j$ (the role of $i,j$ is different from that in the original statement). 
We start from the following computation. 
\begin{eqnarray*}
\Ext^i_{\wS}(\Ext^j_{\wS}(\wS/\wI, \omega_{\wS}), \omega_{\wS})
&\cong&\Ext^i_{\wS}(\Ext^j_{\wS}(\pol(S/I), \omega_{\wS}), \omega_{\wS}) \\
&\cong& \Ext^i_{\wS}(\dpol(\Ext^j_S(S/I, S(-\ba))), \omega_{\wS}) \\
&\cong& \pol(\Ext^i_S(\Ext^j_S(S/I, S(-\ba)), S(-\ba))) \\
&\cong& \pol(\Ext^i_S(\Ext^j_S(S/I, \omega_S), \omega_S)),  
\end{eqnarray*}
where the second and the third isomorphisms follow from \cite[Theorem~4.4]{Y12}. 
Hence we have 
\begin{equation}\label{toriaezuno matome}
\left(\Ext^i_{\wS}(\Ext^j_{\wS}(\wS/\wI, \omega_{\wS}), \omega_{\wS})\right)_{\zero} \cong 
\left(\Ext^i_S(\Ext^j_S(S/I, \omega_S), \omega_S)\right)_{\zero}.
\end{equation}

For a $\bZ^n$-graded module $M=\bigoplus_{\bb \in \bZ^n}M_\bb$, set $M_{\succeq \zero}=\bigoplus_{\bb \succeq \zero}M_\bb$. 
Then, \cite[Theorem~2.3 (1) (a)]{RadicalModules}  states that 
\begin{equation}\label{Ext and radical}
\left(\Ext^j_S(S/I, \omega_S)\right)_{\succeq \zero} \cong \Ext^j_S(S/\sqrt{I}, \omega_S).
\end{equation}

Ene and Okazaki \cite{RadicalModules} constructed two exact functors $\fr^*: \gmod_\ba S \to \Sq S$ and  $\fs^*: \gmod_\ba S \to \Sq S$. 
Here $\fr^*$ is the ``radical functor'' sending a monomial ideal $I$ (resp. $S/I$) to $\sqrt{I}$ (resp. $S/\sqrt{I}$), and $\fs^*$ is just defined by 
$\fs^*(M)=[M(\ba-\mathbf 1)]_{\succeq \zero}$. Now we have 
%$$\fs^*(\Ext^j_R(R/I, R(-\ba))) \cong \Ext^j_R(R/\sqrt{I}, \omega_R).$$ 
\begin{eqnarray*}
\fr^*\left(\Ext^i_S(\Ext^j_S(S/I, \omega_S), \omega_S)\right) & \cong &  \fr^*\left(\Ext^i_S(\Ext^j_S(S/I, S(-\ba)), S(-\ba))\right) \\
 & \cong & \Ext^i_S(\fs^*(\Ext^j_S(S/I, S(-\ba))), \omega_S) \\
 & \cong & \Ext^i_S(\Ext^j_S(S/\sqrt{I}, \omega_S), \omega_S),
\end{eqnarray*}
where the second isomorphism follows from   \cite[Theorem~2.3 (1) (b)]{RadicalModules},  and the third one follows from 
\eqref{Ext and radical} (recall that $\omega_S =S(-{\mathbf 1})$). By the construction of $\fr^*$, we have 
$$\left(\Ext^i_S(\Ext^j_S(S/I, \omega_S), \omega_S)\right)_{\zero} \cong \left(\Ext^i_S(\Ext^j_S(S/\sqrt{I}, \omega_S), \omega_S)\right)_{\zero}.$$
Combining this with \eqref{toriaezuno matome}, we have the expected isomorphism. 
\end{proof}

\begin{remark}
Theorem~\ref{Main Thm Pol LyuNum} is a delicate result  in the following sense. For example, set
$$ I:=(x^2y, x^2z, xyz, xz^2, y^3, y^2z,yz^2) \subset S:=\kk[x,y,z]$$ 
and
$$I' := (x_1x_2y_3, x_1x_2z_3, x_1y_2z_3, x_1z_2z_3, y_1y_2y_3, y_1y_2z_3, y_1z_2z_3) \subset \wS:= \kk[x_i, y_i, z_i | i=1,2,3]. $$
Then  $\Theta$ forms an $\wS/I'$-regular sequence, and gives an isomorphism $\wS/(I'+ (\Theta)) \cong S/I$ (see last part of Subsection \ref{Pre Monomials} to recall notation). 
In this sense, $I'$ is a ``non-standard'' polarization of $I$. We see that $S/I$ is sequentially Cohen-Macaulay, and has a trivial 
Lyubeznik table \cite{AM-SCM,AY-14}. However, we have 
$$
\lambda_{i,j}(\wS/I')=
\begin{cases}
2 & \text{if $i=j=7$,}\\
1 & \text{if $i=5, j=6$,}\\
0 & \text{otherwise.}
\end{cases}
$$
Hence a non-standard polarization may not preserve the Lyubeznik table.  
\end{remark}

\begin{remark}\label{Pol Highest LyuNum}
Under the Notation of Theorem \ref{Main Thm Pol LyuNum}, we can give a different proof of
$$
\lambda_{d,d}\left(S/I\right)= \lambda_{d',d'}\left(\widetilde{S}/\widetilde{I}\right).
$$
by comparing the Hochster-Huneke graph of $S/I$ and $\widetilde{S}/\widetilde{I}.$
The highest Lyubeznik number  for the localization of an Stanley-Reisner ring at the maximal homogeneous ideal is equal to the number of the connected components of the Hochster-Huneke graph by Remark \ref{Rem HH Mon}. 
The vertices of the Hochster-Huneke graph are the minimal primes of minimum height. Let $\{g_1,\dots,g_l\}$ be the set of minimal monomial generators  of $I$ and let $\wg_i$ be the polarization of $m_i$. Then $\{\wg_1,\dots,\wg_l\}$ is the set of minimal monomial generators of $\widetilde{I}$.

We make two observations. First, $(x_{i_1},\ldots,x_{i_r})$ is a vertex in the Hochster graph of $I$ if and only if 
$(x_{i_1,1},\ldots,x_{i_r,1})$ is a vertex of the 
Hochster-Huneke graph of $\widetilde{I}$.
Second, if $(x_{i_1 ,c_1},\dots,x_{i_r,c_r})$ is a vertex of 
the Hochster-Huneke graph of $\widetilde{I}$, then 
$(x_{i_1,b_1},\dots,x_{i_r,b_r})$ is a vertex of 
Hochster-Huneke graph of $\widetilde{I}$ for all $b_j \leq c_j$.

We note that $(x_{i_1},\dots,x_{i_r})$ and $(x_{j_1},\dots,x_{j_r})$ are connected to each other by an edge if and only if $(x_{i_1 ,1},\dots,x_{i_r,1})$ and $(x_{j_1 ,1},\dots,x_{j_r,1})$ are connected to each other by an edge. In addition, $(x_{i_1,1},\dots,x_{i_r,1})$ and $(x_{i_1,c_1},\dots,x_{i_r,c_r})$ are connected via the following paths in Hochster-Huneke graph: 

$$(x_{i_1,1},x_{i_2,1},\ldots,x_{i_r,1}),
(x_{i_1,2},x_{i_2,1},\ldots,x_{i_r,1}),
\ldots,
(x_{i_1,c_1},x_{i_2,1},\ldots,x_{i_r,1})
$$
$$
(x_{i_1,c_1},x_{i_2,1},\ldots,x_{i_r,1}),
(x_{i_1,c_1},x_{i_2,2},\ldots,x_{i_r,1}),
\ldots,
(x_{i_1,c_1},x_{i_2,c_2},\ldots,x_{i_r,1})$$
$$\vdots$$
$$(x_{i_1,c_1},\ldots,x_{i_{r-1},c_{r-1}},x_{i_r,1}),
(x_{i_1,c_1},\ldots,x_{i_{r-1},c_{r-1}},x_{i_r,2}),
\ldots,
(x_{i_1,c_1},\ldots,x_{i_{r-1},c_{r-1}},x_{i_r,c_r}).
$$

By the previous assertions, we get that there is an bijective correspondence between connected components of the two Hochster-Huneke graphs under discussion, which sends the connected component of $(x_{i_1,c_1},\dots,x_{i_r,c_r})$ to the connected component of $(x_{i_1},\dots,x_{i_r})$.
 \end{remark}

\begin{proposition}\label{Prop Pol Gral LyuNum}
Let $S=K[x_1,\ldots,x_n]$ be a polynomial ring and $I\subset S$ 
be a monomial ring. Let $\widetilde{I}$ denote the polarization of $I,
$ and $\widetilde{S}=K[x_{r,s}]$ denote the polarization ring. 
Let $h=\dim(\widetilde{S}/\widetilde{I})-\dim(S/I)$.
Then,
$$
\lambda^{j-h}_{0}\left(S/\sqrt{I}\right)
\leq \lambda^{j}_0\left(\widetilde{S}/\widetilde{I}\right).
$$
%for $j=\Depth_I{S},\ldots,\dim(R)$, and all the other vanish.
Moreover,
$$
\gamma_{i-h,j-h}(S/I)
\leq \gamma_{i,j}\left(\widetilde{S}/\widetilde{I}\right).
$$
\end{proposition}
\begin{proof}
We fix $j\in\NN.$
Let $\m=(x_1,\ldots,x_n)$ denote the 
maximal homogeneous ideal of $R.$
Since $I$ is a monomial ideal, there exists a filtration of
$D_K(S)$-modules
$$
0=M_0\subset M_1\subset M_2\subset \ldots \subset M_\ell=H^{n-j}_I(S)
$$
such that $M_{t}/M_{t-1}\cong H^{\Ht(\fp_t)}_{\fp_t}(S)$ 
where $\fp_t$ is a prime ideal generated by a subset of $\{x_1,\ldots,x_n\}$
\cite[Proposition $2.10$ and Remark $2.12$]{YStr}.
We recall that 
$$
m_{n-j,-\alpha}=\#
\{t|M_{t}/M_{t-1}\cong H^{|\alpha|}_{\fp_{\alpha}}(S)\}.$$
We note that $H^{\Ht(\fp_t)}_{\fp_t}(R)$ is a simple $D_K(S)$ for every $t=1,\ldots,\ell,$
and so, $\Length_{D_K(S)}H^{n-j}_I(S)=\sum_{\alpha}m_{n-j,-\alpha}=\ell.$
Let $Q=m\widetilde{S},$ $A=\widehat{\widetilde{S}_{Q}}$ and $L=A/{Q}A.$
We note that the induced map $S\to \widetilde{S}\to A$ makes $A$ a flat extension of $S.$
We have that 
$$
H^{n-j}_I(S)\otimes_S A\cong H^{n-j}_{I}(A)
$$ 
is a $D_L (A)$-module.
In addition, 
$$
0=M_0\otimes_S A\subset M_1\otimes_S A\subset M_2\otimes_S A\subset \ldots \subset M_\ell=H^{n-j}_I(A),
$$
gives a filtration of $D_F(A)$-modules such that
$M_t \otimes_S A/M_{t-1}\otimes_S A\cong H^{\Ht{\fp_t}}_{\fp_t}(A)\neq 0$ is a simple $D_L(A)$-module.
Hence, 
$$m_{n-i,-\alpha}=\#
\{t|M_{t}/M_{t-1}\otimes_S A\cong H^{|\alpha|}_{\fp_{\alpha}}(A)\}.$$

As before, there exists a filtration
of $D_K (\widetilde{S})$-modules
$$
0=M'_0\subset M'_1\subset M'_2\subset \ldots \subset M'_{\ell'}=H^{n-j}_{\widetilde{I}}(\widetilde{S})
$$
such that $M'_{t}/M'_{t-1}\cong H^{\Ht \fp'_t}_{\fp'_t}(\widetilde{S})$
where $\fp'_t$ is a prime ideal generated by a subset of $\{x_{r,s}\}$
\cite[Proposition $2.10$ and Remark $2.12$]{YStr}.
We note that $H^{\Ht{\fp'_t}}_{\fp'_t}(\widetilde{S})$ is a simple $D_L(\widetilde{S})$-module,
and so, $\Length_{D_K(\widetilde{S})}H^i_{\widetilde{I}}(\widetilde{S})=\ell'.$
We have that $A$ is a flat $\widetilde{S}$-algebra and
$$
H^i_{\widetilde{I}}(\widetilde{S})\otimes_{\widetilde{S}} A\cong H^{n-j}_{I}(A)
$$ 
is a $D_L( A)$-module.
As before, we have that 
$$
0=M'_0\otimes_{\widetilde{S}} A\subset M'_1\otimes_{\widetilde{S}} A\subset M'_2\otimes_{\widetilde{S}} A\subset \ldots \subset M'_{\ell'}\otimes_{\widetilde{S}} A=H^i_{\widetilde{I}}(A),
$$
gives a filtration of $D_L(A)$-modules.
We have that
$M'_{t\otimes_S A}/M'_{t-1}\otimes_{\widetilde{S}} A\neq 0$ 
if and only if $\fp'_t\subset \{x_{1,1},\ldots,x_{n,1}\},$ and if that is the case,  $M'_{t}/M'_{t-1}\otimes_{\widetilde{S}} A\cong H^{\mid \fp'_t\mid}_{\fp'_t}(A)$ is a simple $D_L (A)$-module. 
Hence,
if $m'_{-\alpha}=\#
\{t|M_{t}/M_{t-1}\otimes_S A\cong H^{|\alpha|}_{\fp_{\alpha}}(A)\hbox{ and } \fp_\alpha \subset Q\},$
we have that $m'_{-\alpha}\leq m_{n-j,-\alpha}.$
Hence, 
\begin{align*}
\lambda^{j+h}_{0}(\widetilde{S}/\widetilde{I})&= \Length_{D_K (\widetilde{S})}H^{\dim(\widetilde{S})-j-h}_{I}(\widetilde{S})\\
&=\Length_{D_L (\widetilde{S})}H^{n-j-h}_{\widetilde{I}}(\widetilde{S})\\
&= \sum_{\alpha} m'_{-\alpha}\\
& \leq  \sum_{\alpha} m_{n-j,-\alpha}\\
& =\Length_{D_K (S)}H^{n-j}_{I}(S)\\
& =\lambda^{j}_{0}(S_\m/IS_\m).
\end{align*}

Moreover, 
$$
\gamma_{i+h,j+h}(\widetilde{S}/\widetilde{I}) =\sum_{|\alpha|=n-j} m'_{-\alpha} 
\leq \sum_{|\alpha|=n-j} m_{n-j,-\alpha}=
\gamma_{i,j}(S/I).
$$
\end{proof}

\begin{remark}
As noted in Remark \ref{RemMultBetti}, the generalized Lyubeznik numbers and invariants $\gamma_{i,j}$ are given by adding multigraded Betti numbers of $I^\vee.$
A corollary of the proof for Proposition \ref{Prop Pol Gral LyuNum} is that
$$
\beta_{i,-\alpha}\left(\left(\sqrt{I}\right)^\vee\right)\leq \beta_{i,-(\alpha,0,\ldots,0)}\left(\left(\widetilde{I}\right)^\vee\right),
$$
where $(\alpha,0,\ldots,0)$ is the extension of the exponent $\alpha$
given the inclusion $S\to \widetilde{S}$ given by $x_j\mapsto x_{j,1}.$
We point out that this corollary is not a consequence of the fact that
$$
\beta_{i,-\alpha}(\left(\sqrt{I}\right)\leq
\beta_{i,-\alpha}(\left(I\right)= 
\beta_{i,-(\alpha,0,\ldots,0)}(\left(\widetilde{I}\right),
$$
which follows from \cite[Theorem $1.4$]{RadicalModules} and \cite[Theorem $2.14$]{Peeva}.
\end{remark}

\section{Acknowledgments}
We thank Josep \`Alvarez-Montaner, Craig Huneke, and Ryota Okazaki
for helpful comments and suggestions.
The first author was partially supported by NSF Grant DMS-$1259142.$
The second author was partially supported by the Council for Science and Technology of Mexico, CONACYT $207063.$ The third author was partially supported
by JSPS KAKENHI $25400057.$

\bibliographystyle{alpha}
\bibliography{References}

{\footnotesize

\noindent \small \textsc{Department of Mathematics, University of Virginia, Charlottesville, VA  22904-4137} \\ \indent \emph{Email address}:  {\tt ab4cb@virginia.edu} 

\vspace{.25cm}

\noindent \small \textsc{Department of Mathematics, University of Virginia, Charlottesville, VA  22904-4135} \\ \indent \emph{Email address}:  {\tt lcn8m@virginia.edu} 
 
\vspace{.25cm}

\noindent \small \textsc{Department of Mathematics, Kansai university, Suita 564-8680, Japan} \\ \indent \emph{Email address}:  {\tt yanagawa@ipcku.kansai-u.ac.jp} 
}

\end{document}